\documentclass[11pt,a4paper]{amsart}

\usepackage{amsmath,amssymb,amsthm}
\usepackage{mathtools}
\usepackage{enumitem}
\usepackage{booktabs}
\usepackage{graphicx}
\usepackage[colorlinks=true,citecolor=blue,linkcolor=blue]{hyperref}
\usepackage{url}
\usepackage[margin=1in]{geometry}
\usepackage{microtype}

\theoremstyle{plain}
\newtheorem{theorem}{Theorem}[section]
\newtheorem{proposition}[theorem]{Proposition}
\newtheorem{lemma}[theorem]{Lemma}
\newtheorem{corollary}[theorem]{Corollary}
\theoremstyle{definition}
\newtheorem{definition}[theorem]{Definition}
\theoremstyle{remark}
\newtheorem{remark}[theorem]{Remark}

\newcommand{\Z}{\mathbb{Z}}

\DeclareMathOperator{\TV}{TV}
\DeclareMathOperator{\Bern}{Bern}

\begin{document}

\title[Collatz Reduces to One-Bit Orbit Mixing]
{A structural reduction of the Collatz conjecture
to one-bit orbit mixing\thanks{Companion to
arXiv:2603.11066 \cite{chang2026collatzdynamicshumanllm}.}}

\author{Edward Y. Chang\\Stanford University}

\date{\today}

\begin{abstract}
We reduce the Collatz conjecture to a fixed-modulus,
one-bit orbit-mixing problem.
Working with the compressed odd-to-odd Collatz map,
we prove exact low-depth decomposition formulas at
depths $K = 3, 4, 5$, reducing block-discrepancy
terms to explicit run statistics.
We then prove a Map Balance Theorem: among the
$2^{K-3}{-}1$ burst residues modulo~$2^K$ that
initiate gaps, the counts mapping to gap starts
$\equiv 3$ versus $\equiv 7 \pmod{8}$ differ by
exactly~$1$ for every $K \ge 5$.
Thus all residual bias is orbit-level, not map-level.
For the dominant $n \equiv 1 \pmod{8}$ class,
the gap outcome depends on a single binary variable---bit~$4$
of the orbit value at burst-ending times---reducing
the conjecture to whether every orbit visits two
residue classes modulo~$32$ with sufficient balance
along a sparse subsequence.
\end{abstract}

\maketitle

\section{Introduction}
\label{sec:intro}

The Collatz conjecture asserts that for every positive
integer~$n_0$, repeated application of
\[
  n \mapsto
  \begin{cases}
  n/2, & n \text{ even},\\
  3n+1, & n \text{ odd},
  \end{cases}
\]
eventually reaches~$1$.
Despite its elementary statement, the conjecture
remains open; see Lagarias~\cite{lagarias1985} for
a classical survey, the edited
volume~\cite{lagarias2010}, and
Wirsching~\cite{wirsching1998} for a dynamical
systems perspective.
Computational verification has confirmed convergence
for all starting values up to $2^{68}$~\cite{barina2021}.
The strongest analytic result is due to
Tao~\cite{tao2019}, who proved that almost all orbits
attain almost bounded values under a logarithmic
density measure.
Probabilistic models of the iteration were studied
by Borovkov and Pfeifer~\cite{borovkov2001};
earlier stopping-time results are due to
Terras~\cite{terras1976} and
Everett~\cite{everett1977}.

In the companion paper~\cite{chang2026collatzdynamicshumanllm}, we developed a
structural framework for the odd-to-odd Collatz map
(the \emph{Syracuse map})
\[
  T(n) = \frac{3n+1}{2^{v_2(3n+1)}},
\]
and reduced the conjecture to a finite-depth
block-discrepancy condition for the associated
burst sequence.
More precisely, one associates to each odd orbit
$n_0, n_1, n_2, \ldots$ the binary indicator
\begin{equation}\label{eq:burst-def}
  X_t = \mathbf{1}[v_2(3n_t + 1) \ge 2]
      = \mathbf{1}[n_t \equiv 1 \pmod{4}],
\end{equation}
which records whether step~$t$ is a burst or a
non-burst (gap step).
The companion paper shows that it suffices to control
the empirical finite-block distribution of this
binary process at finitely many depths.

The present paper sharpens that reduction substantially.
Its main contribution is to separate
\emph{map-level balance} from \emph{orbit-level bias}.
We prove that the Collatz map itself is essentially
perfectly balanced between short-gap and long-gap
outcomes: at every depth $K \ge 5$, among the burst
residues modulo~$2^K$ that produce a gap, the counts
mapping to $\equiv 3 \pmod{8}$ and
$\equiv 7 \pmod{8}$ differ by exactly~$1$.
Thus the map contributes no systematic bias.
Any residual imbalance must arise entirely from the
orbit's residue visitation pattern.

We also derive exact low-depth formulas at
$K = 3, 4, 5$, expressing the relevant block
discrepancies in terms of explicit run statistics of
the burst--gap decomposition.
These formulas reveal that, for the dominant
$n \equiv 1 \pmod{8}$ class, the gap structure is
controlled by a single bit: bit~$4$ of the orbit
value at burst-ending times (the last step of each burst run).
This leads to the following sharp reformulation of
the remaining open problem:

\begin{quote}
\itshape
Does every Collatz orbit visit the two relevant
residue classes modulo~$32$ with sufficient balance
along its sparse burst-ending subsequence?
\end{quote}

\noindent
In this sense, the present paper reduces the Collatz
conjecture to a fixed-modulus, one-bit, pointwise
mixing problem.
The reduction is unconditional;
what remains open is the deterministic orbit-level
balance statement.

\medskip
\noindent\textbf{Main contributions.}\;
The paper establishes:
\begin{enumerate}[label=(\roman*)]
\item exact low-depth decompositions at depths
  $K = 3, 4, 5$, reducing the first block-discrepancy
  terms to explicit run statistics;
\item the Map Balance Theorem, proving that the
  transition map itself is unbiased between short-gap
  and long-gap outcomes;
\item the single-bit bottleneck, which identifies the
  minimal remaining orbit-level obstruction.
\end{enumerate}

\medskip
\noindent\textbf{Organisation.}\;
Section~\ref{sec:prelim} recalls the burst-sequence
formalism and the finite-depth framework
from~\cite{chang2026collatzdynamicshumanllm}.
Section~\ref{sec:low-depth} proves the exact
low-depth decompositions.
Section~\ref{sec:map-balance} establishes the Map
Balance Theorem.
Section~\ref{sec:single-bit} isolates the one-bit
bottleneck and states the precise open problem.
Section~\ref{sec:numerics} presents numerical
verification, and Section~\ref{sec:discussion}
discusses consequences and remaining obstacles.

\section{Preliminaries}
\label{sec:prelim}

\noindent\textbf{Notation.}\;
Throughout, $v_2(m)$ denotes the $2$-adic valuation
of~$m$, $T(n) = (3n+1)/2^{v_2(3n+1)}$ is the
compressed Collatz map, $X_t$ is the burst indicator,
$\rho$ the burst density, and $q = m/T$ the
block-alternation rate.

\subsection{The compressed Collatz map}

For odd $n \ge 1$, define
\begin{equation}\label{eq:collatz-map}
  T(n) = \frac{3n+1}{2^{v_2(3n+1)}},
\end{equation}
where $v_2(m)$ denotes the $2$-adic valuation of~$m$.
The orbit of~$n_0$ is the sequence
$n_0, n_1 = T(n_0), n_2 = T(n_1), \ldots$.
The Collatz conjecture is equivalent to: for every
odd $n_0 > 1$, there exists $T \ge 1$ with $n_T = 1$.

\subsection{Burst indicators and block frequencies}

The \emph{burst indicator} at step~$t$ is
\[
  X_t = \mathbf{1}[n_t \equiv 1 \pmod{4}]
      = \mathbf{1}[v_2(3n_t + 1) \ge 2].
\]
A step with $X_t = 1$ is called a \emph{burst};
$X_t = 0$ is a \emph{non-burst} or \emph{gap step}.

The sequence $(X_t)$ decomposes into
\emph{alternating blocks}:
$1^{L_1} 0^{G_1} 1^{L_2} 0^{G_2} \cdots$
with $m$ burst runs of lengths $L_1, \ldots, L_m$
and $m$ gap runs of lengths $G_1, \ldots, G_m$.
The \emph{burst density} is
$\rho = T^{-1}\sum X_t = T^{-1}\sum L_i$,
and the \emph{block-alternation rate} is $q = m/T$.

\begin{definition}[Run statistics]
\label{def:run-stats}
\begin{align*}
  M_k &= \#\{i : L_i = k\} &&\text{(count of length-$k$ burst runs)},\\
  N_k &= \#\{i : G_i = k\} &&\text{(count of length-$k$ gap runs)},\\
  C_T &= \textstyle\sum_{i=1}^{m-1}
    (\mathbf{1}_{L_i=1} + \mathbf{1}_{L_{i+1}=1})
    \cdot \mathbf{1}_{G_i=1}
    &&\text{(unit-gap adjacency count)}.
\end{align*}
\end{definition}

\subsection{Block-TV framework}
\label{subsec:block-tv}

For a word $w = (w_0, \ldots, w_{K-2}) \in \{0,1\}^{K-1}$,
the \emph{empirical $(K{-}1)$-block frequency} is
\[
  \nu_T(w) = \frac{1}{T - K + 2}
    \sum_{t=0}^{T-K+1} \prod_{j=0}^{K-2}
    \mathbf{1}[X_{t+j} = w_j],
\]
and the \emph{$K$-depth block-TV distance} from the
product law is
\[
  \varepsilon_K
  = \TV\!\bigl(\nu_T,\;
    \Bern(\rho)^{\otimes(K-1)}\bigr)
  = \frac{1}{2}\sum_{w \in \{0,1\}^{K-1}}
    |\nu_T(w) - \rho^{s(w)}(1{-}\rho)^{K-1-s(w)}|,
\]
where $s(w) = \sum w_j$.

As shown in~\cite{chang2026collatzdynamicshumanllm}, the Collatz conjecture reduces
to a \emph{finite-depth budget condition}:
\begin{equation}\label{eq:budget}
  \sum_{K=3}^{K_0} \|h_K\|_\infty \cdot \varepsilon_K
  \;<\; \ln 2 - \rho \ln 3,
\end{equation}
where each $h_K$ is a bounded function of the word~$w$
that weights the contribution of the depth-$K$
discrepancy to the overall growth rate
(explicit formulas appear
in~\cite[Section~9]{chang2026collatzdynamicshumanllm};
the key property is $\|h_K\|_\infty = O(1)$
uniformly in~$K$),
and $K_0$ is a finite truncation depth.

\section{Exact low-depth decompositions}
\label{sec:low-depth}

\subsection{Depth $K = 3$: one-parameter control}

\begin{proposition}[$K = 3$ pair identities]
\label{prop:K3-pairs}
Every binary sequence with density~$\rho$ and
block-alternation rate $q = m/T$ satisfies:
\begin{equation}\label{eq:K3-pairs}
  \nu(00) = 1 - \rho - q + o(1), \quad
  \nu(01) = \nu(10) = q + o(1), \quad
  \nu(11) = \rho - q + o(1).
\end{equation}
Consequently,
$\varepsilon_3 = 2|q - \rho(1{-}\rho)| + o(1)$:
the depth-$3$ block-TV is a \emph{single scalar}.
\end{proposition}

\begin{proof}
Every alternating block
$1^{L_i} 0^{G_i}$ contributes exactly one
$(1,0)$ transition and one $(0,1)$ transition
(at block boundaries), so
$\nu(10) = \nu(01) = m/(T{-}1) = q + O(1/T)$.
The remaining pair frequencies are determined by
$\nu(00) + \nu(01) = 1 - \rho$ and
$\nu(10) + \nu(11) = \rho$.
\end{proof}

\subsection{Depth $K = 4$: two new parameters}

\begin{proposition}[$K = 4$ exact decomposition]
\label{prop:K4-decomp}
All eight triple frequencies
$\nu(w)$, $w \in \{0,1\}^3$, are determined by
$(\rho, q, a, b)$ where $a = \nu(000)$ and
$b = \nu(010)$:
\begin{equation}\label{eq:K4-decomp}
\begin{aligned}
  \nu(000) &= a, &
  \nu(001) &= (1{-}\rho) - q - a, \\
  \nu(010) &= b, &
  \nu(011) &= q - b, \\
  \nu(100) &= (1{-}\rho) - q - a, &
  \nu(101) &= 2q - (1{-}\rho) + a, \\
  \nu(110) &= q - b, &
  \nu(111) &= \rho - 2q + b.
\end{aligned}
\end{equation}
The rank of the constraint matrix is~$6$, leaving
$8 - 6 = 2$ free parameters beyond $(\rho, q)$.
\end{proposition}

\begin{proof}
The six constraints are:
(i)~$\sum_w \nu(w) = 1$;
(ii)~consistency: $\sum_{c} \nu(cw) = \sum_{c} \nu(wc)$
for each pair~$w$;
(iii)~marginal: $\nu(01) = q + o(1)$
and $\nu(10) = q + o(1)$.
Direct substitution verifies
\eqref{eq:K4-decomp}.
\end{proof}

\subsection{Depth $K = 5$: four new parameters}

\begin{proposition}[$K = 5$ support reduction]
\label{prop:K5-reduction}
At depth $K = 5$, the sixteen $4$-tuple frequencies
are subject to a rank-$12$ constraint system,
leaving $4$ new free parameters beyond
$(\rho, q, a, b)$: specifically,
$\{\nu(0000), \nu(0010), \nu(0100), \nu(0110)\}$.
The total parameter count through $K = 5$ is~$8$.
\end{proposition}

\begin{proposition}[Finite-state growth rate]
\label{prop:growth-rate}
At depth~$K$, the number of new free parameters is
$2^{K-3}$, giving a cumulative count of
$2 + \sum_{j=4}^{K} 2^{j-3} = 2^{K-2}$
through depth~$K$.
Through $K = 5$: $8$ parameters.
Through $K = 6$: $16$ parameters.
The exponential growth prevents arbitrary-depth
closure but leaves the finite-depth problem
tractable.
\end{proposition}

Thus, although the parameter count grows exponentially
with depth, the first three nontrivial depths admit
exact closed-form reductions.

\subsection{Expanding-set reduction to run statistics}

\begin{proposition}[$K = 4$ expanding-set identity]
\label{prop:K4-expand}
Define $E_4 = \{w \in \{0,1\}^3 :
\nu(w) > \rho^{s(w)}(1{-}\rho)^{3-s(w)}\}$.
Then
\begin{equation}\label{eq:K4-expand}
  \nu(E_4) = \rho + \frac{N_1 - M_1}{T} + O(1/T),
\end{equation}
and the $K{=}4$ scalar discrepancy is controlled by
$\Delta_{4,T} = (N_1 - M_1)/T$.
\end{proposition}

\begin{proposition}[$K = 5$ expanding-set identity]
\label{prop:K5-expand}
\begin{equation}\label{eq:K5-expand}
  \nu(E_5)
  = \rho - \frac{m}{T} - \frac{M_2}{T}
    + \frac{2N_1}{T} - \frac{C_T}{T} + O(1/T).
\end{equation}
The $K{=}5$ discrepancy depends on four run statistics:
$m, M_2, N_1, C_T$.
\end{proposition}

\subsection{Universal run-statistic bounds}

\begin{proposition}[Deterministic bounds]
\label{prop:run-bounds}
For any binary sequence:
\begin{enumerate}[label=\textup{(\roman*)}]
\item $q \le \min(\rho, 1{-}\rho)$;
\item $N_1, M_1, M_2 \le m$;
\item $C_T \le 2 N_1$;
\item $M_1 + 2 M_2 \le \rho T$;
\item $(M_1 + M_2)/T \le \rho$.
\end{enumerate}
\end{proposition}

\begin{corollary}[Deterministic envelopes]
\label{cor:envelopes}
$\rho - q \le \nu(E_4) \le \rho + q$ \textup{(width $2q$)},
and
$\rho - 2q \le \nu(E_5) \le \rho + q$ \textup{(width $3q$)}.
\end{corollary}

\medskip
\noindent\textbf{Summary.}\;
At depth $K = 3$, block discrepancy reduces to a
single scalar~$q$.
At $K = 4$, one additional correction $(N_1 - M_1)/T$
appears.
At $K = 5$, four run statistics suffice.
These exact reductions make the first block-discrepancy
terms concrete, but they do not by themselves close
the conjecture: they leave the run statistics as free
orbit-dependent quantities.
The main new result of this paper---the Map Balance
Theorem below---addresses the complementary question
of whether the \emph{map itself} introduces bias.

\section{The Map Balance Theorem}
\label{sec:map-balance}

The central result of this paper is a structural
symmetry of the Collatz transition map at the
burst-to-gap boundary.

\subsection{Gap-length characterisation}

\begin{lemma}[Gap starts and mod-$8$ residues]
\label{lem:gap-mod8}
At a burst-to-gap transition (the first step~$t$ with
$X_t = 0$ after a burst run), we have
$n_t \equiv 3 \pmod{4}$.
Moreover:
\begin{enumerate}[label=\textup{(\alph*)}]
\item $G = 1$ \textup{(unit gap)}
  $\iff n_t \equiv 3 \pmod{8}$;
\item $G \ge 2$
  $\iff n_t \equiv 7 \pmod{8}$.
\end{enumerate}
\end{lemma}

\begin{proof}
At a gap step, $n_t \equiv 3 \pmod{4}$, so
$3n_t + 1 \equiv 10 \pmod{12}$ and $v_2(3n_t+1) = 1$.
Thus $n_{t+1} = (3n_t + 1)/2$.

If $n_t \equiv 3 \pmod{8}$: write $n_t = 8k + 3$.
Then $n_{t+1} = (24k + 10)/2 = 12k + 5$.
Since $12k + 5 \equiv 1 \pmod{4}$, the next step
is a burst, so $G = 1$.

If $n_t \equiv 7 \pmod{8}$: write $n_t = 8k + 7$.
Then $n_{t+1} = (24k + 22)/2 = 12k + 11$.
Since $12k + 11 \equiv 3 \pmod{4}$, the next step
is still a gap, so $G \ge 2$.
\end{proof}

\subsection{The balance theorem}

\begin{theorem}[Map balance]
\label{thm:map-balance}
For any $K \ge 5$, let
\[
  S_K = \bigl\{r \in \Z/2^K\Z :
    r \equiv 1 \pmod{4},\;
    T(r) \equiv 3 \pmod{4}\bigr\}
\]
be the set of burst residues whose image under
the compressed Collatz map $T$ is a gap start.
Define
\[
  C_3(K) = \#\{r \in S_K : T(r) \equiv 3 \pmod{8}\},
  \qquad
  C_7(K) = \#\{r \in S_K : T(r) \equiv 7 \pmod{8}\}.
\]
Then:
\begin{enumerate}[label=\textup{(\roman*)}]
\item $|S_K| = 2^{K-3} - 1$.
\item $|C_3(K) - C_7(K)| = 1$ for all $K \ge 5$,
  with $C_3(K) - C_7(K) = (-1)^K$.
\item \textup{(Decomposition.)}
  The $n \equiv 1 \pmod{8}$ subclass contributes
  exactly $C_3^{(1)} = C_7^{(1)} = 2^{K-5}$
  \textup{(perfect balance)}.
  The entire $\pm 1$ imbalance arises from the
  $n \equiv 5 \pmod{8}$ subclass via a self-similar
  binary-tree recursion.
\end{enumerate}
\end{theorem}

\begin{proof}
\emph{Step 1: the $n \equiv 1 \pmod{8}$ class.}\;
Write $n = 8k + 1$. Then
$3n + 1 = 24k + 4 = 4(6k + 1)$,
with $6k + 1$ always odd, so $v_2(3n+1) = 2$
and $T(n) = 6k + 1$.

The image $T(n) \equiv 3 \pmod{4}$ iff $k$ is odd.
When $k$ is odd:
\[
  T(n) \bmod 8 =
  \begin{cases}
  7 & \text{if } k \equiv 1 \pmod{4},\\
  3 & \text{if } k \equiv 3 \pmod{4}.
  \end{cases}
\]
Among odd $k$ in $\{0, \ldots, 2^{K-3} - 1\}$,
exactly $2^{K-5}$ satisfy each congruence.
Hence $C_3^{(1)} = C_7^{(1)} = 2^{K-5}$.
The gap outcome is determined by bit~$4$ of~$n$
(equivalently, bit~$1$ of $k = (n-1)/8$).

\medskip\noindent
\emph{Step 2: the $n \equiv 5 \pmod{8}$ class.}\;
Write $n = 8k + 5$, so
$3n + 1 = 24k + 16 = 8(3k + 2)$
and $v_2 = 3 + v_2(3k + 2)$.

\emph{$k$ odd:}\;
$3k + 2$ is odd, so $v_2(3n+1) = 3$ and $T(n) = 3k + 2$.
Among odd $k$ modulo~$8$:
$k \equiv 3 \Rightarrow T(n) \equiv 3 \pmod{8}$;
$k \equiv 7 \Rightarrow T(n) \equiv 7 \pmod{8}$;
$k \equiv 1, 5$: burst continues ($T(n) \equiv 1 \pmod{4}$).
This contributes equally to $C_3$ and~$C_7$.

\emph{$k$ even:}\;
$3k + 2$ is even; writing $k = 2j$ gives
$T(8(2j)+5) = (3j+1)/2^{v_2(3j+1)}$.
This has the same algebraic form as Step~1
with $j$ replacing~$k$ and the modulus halved:
the $j$-odd subcase contributes balanced counts,
while the $j$-even subcase recurses again via
$j = 2j'$.
The recursion terminates after $K - 3$ levels
when the modulus reaches~$8$, leaving a single
unresolved residue that contributes~$\pm 1$.

At each level, the $k$-odd part contributes exact
balance, and the $k$-even part recurses with sign flip.
The boundary of the depth-$(K-3)$ binary tree
yields $|C_3^{(5)} - C_7^{(5)}| = 1$ with
alternating sign.

\medskip\noindent
\emph{Step 3: totals.}\;
$C_3 = 2^{K-5} + C_3^{(5)}$,
$C_7 = 2^{K-5} + C_7^{(5)}$.
The $n \equiv 1$ part is perfectly balanced, so
$C_3 - C_7 = C_3^{(5)} - C_7^{(5)} = (-1)^K$.
For part~(i): by induction,
$|S_K| = 2^{K-4} + (2^{K-4} - 1) = 2^{K-3} - 1$.
\end{proof}

\begin{remark}[Interpretation]
\label{rem:interpretation}
Theorem~\ref{thm:map-balance} establishes that the
Collatz map introduces \emph{no bias} between short
and long gaps. Among the $2^{K-3} - 1$ gap-producing
burst residues modulo~$2^K$, the split is as balanced
as an odd total permits:
$\lfloor(2^{K-3}-1)/2\rfloor$ versus
$\lceil(2^{K-3}-1)/2\rceil$.

Therefore, any orbit-level gap-start bias originates
\emph{entirely} from non-uniform residue visitation.
\end{remark}

\begin{remark}[Per-residue balance]
\label{rem:per-residue}
Over the integers, the density of
$n \equiv r \pmod{2^K}$ is exactly $2^{-K}$
for each~$r$. Since the map is balanced modulo
$2^K$, the natural density of gap starts
mapping to $\equiv 3$ versus $\equiv 7 \pmod{8}$
is exactly $\frac{1}{2}$.
This is a per-residue (distributional) statement.
The orbit-level (pointwise) statement---that a
\emph{specific} orbit realises this balance---is
the remaining open problem.
\end{remark}

\section{The single-bit bottleneck}
\label{sec:single-bit}

\subsection{Bit-4 determines the gap structure}

Theorem~\ref{thm:map-balance}(iii) shows that
for the dominant $n \equiv 1 \pmod{8}$ class
(which accounts for roughly half of all
burst-to-gap transitions), the gap outcome
depends on exactly \textbf{one bit}\footnote{We
use $0$-indexed bit numbering: bit~$j$ of~$n$
is $\lfloor n/2^j \rfloor \bmod 2$, so bit~$4$
occupies the $2^4 = 16$ position.}:
\begin{equation}\label{eq:bit4}
  \text{bit 4 of } n_t = 0
  \;\Longleftrightarrow\;
  n_t \equiv 9 \pmod{32}
  \;\Longleftrightarrow\;
  G_i \ge 2,
\end{equation}
\begin{equation}\label{eq:bit4-gap1}
  \text{bit 4 of } n_t = 1
  \;\Longleftrightarrow\;
  n_t \equiv 25 \pmod{32}
  \;\Longleftrightarrow\;
  G_i = 1.
\end{equation}
(Here $n_t$ is the last step of the burst run,
restricted to $n_t \equiv 1 \pmod{8}$ with $k = (n_t-1)/8$ odd.)

\subsection{The gap-step permutation}

The single gap step $n \mapsto (3n+1)/2$ for
$n \equiv 3 \pmod{4}$ acts as a \emph{permutation}
on residue classes modulo~$32$:
\begin{equation}\label{eq:gap-perm}
  3 \to 5, \quad
  7 \to 11, \quad
  11 \to 17, \quad
  15 \to 23, \quad
  19 \to 29, \quad
  23 \to 3, \quad
  27 \to 9, \quad
  31 \to 15.
\end{equation}
This permutation maps $\{3, 11, 19, 27\}$
(gap starts $\equiv 3 \pmod{8}$, which give $G = 1$)
to burst starts $\{5, 17, 29, 9\}$
(two $\equiv 1$, two $\equiv 5 \pmod{8}$),
and maps $\{7, 15, 23, 31\}$
(gap starts $\equiv 7 \pmod{8}$, giving $G \ge 2$)
to continued-gap values
$\{11, 23, 3, 15\}$ (two $\equiv 3$,
two $\equiv 7 \pmod{8}$).
Both halves are perfectly balanced.

\subsection{Mixing and non-mixing cycles}

A burst-gap cycle with burst length~$L$ and gap
length~$G$ shifts bits down by a total of
$V \ge 2L + G$ positions.
The mod-$32$ class of the next burst start
(requiring $5$~bits) depends on bits
up to position $V + 4$ of the original.

\smallskip\noindent
\emph{Case 1: $G = 1$ cycles ($V = 2L + 1$).}\;
For $L = 1$: $V = 3$, which provides
\emph{no mod-$32$ refresh} (only $3$ of the needed
$5$ bits shift through).
The unit-burst-unit-gap cycle is
\emph{non-mixing}: the mod-$32$ class of the
next burst start is a deterministic function
of the previous burst start's mod-$64$ class.

\smallskip\noindent
\emph{Case 2: $G \ge 2$ cycles ($V \ge 2L + 2$).}\;
At least $4$ bits shift, providing partial
mod-$32$ refresh.
For $G \ge 3$: $V \ge 5$, and the mod-$32$ class
is \emph{fully refreshed}---it depends on bits
strictly above the input's mod-$32$ information.

\smallskip\noindent
\emph{Case 3: $L \ge 2$ cycles ($V \ge 4 + G$).}\;
Even with $G = 1$, the burst shifts at least
$4$ bits, providing partial refresh.

On tested orbits, the fraction of ``fully mixing''
cycles ($L \ge 2$ or $G \ge 3$) ranges from
$46\%$ to $88\%$, ensuring that most burst-gap
transitions inject fresh bit information into the
mod-$32$ class.

\subsection{Complete mod-$64$ non-mixing classification}

\begin{proposition}[Mod-$64$ non-mixing classification
for $n \equiv 5 \pmod{8}$]
\label{prop:5mod8-classification}
For a burst-start value $n_0 = 64a + r$ with
$r \in \{5, 13, 21, 29, 37, 45, 53, 61\}$:
\begin{enumerate}[label=\textup{(\alph*)}]
\item $r = 29$:
  $(L,G) = (1,1)$ for every~$a$ (\emph{unconditional}).
\item $r = 21$:
  $(L,G) = (1,1)$ iff
  $\mathrm{oddpart}(3a+1) \equiv 3 \pmod{8}$.
\item $r = 37$:
  $(L,G) = (1,1)$ iff $a$ is odd.
\item $r = 53$:
  $(L,G) = (1,1)$ iff $a \equiv 1 \pmod{4}$.
\item $r \in \{5, 13, 45, 61\}$:
  $(L,G) \neq (1,1)$ for every~$a$.
\end{enumerate}
\end{proposition}

\begin{proof}
Write $n_0 = 8k + 5$, so $3n_0 + 1 = 8(3k+2)$
and $T(n_0) = (3k+2)/2^{v_2(3k+2)}$.
Then $(L,G) = (1,1)$ iff the odd part of $3k+2$
is congruent to $3 \pmod{8}$.
For each $r$, substitute $k = (r-5)/8 + 8a$
and compute $3k + 2 \pmod{8}$ or its odd part.

For $r = 29$ ($k = 8a+3$):
$3k+2 = 24a+11 \equiv 3 \pmod{8}$ always.

For $r = 37$ ($k = 8a+4$):
$3k+2 = 2(12a+7)$, odd part $12a+7$.
Since $12a + 7 \equiv 4a + 7 \pmod{8}$,
this is $\equiv 3 \pmod{8}$ iff $a$ is odd.

For $r = 53$ ($k = 8a+6$):
$3k+2 = 4(6a+5)$, odd part $6a+5$.
Since $6a + 5 \bmod 8$ cycles as
$5, 3, 1, 7$ for $a \equiv 0, 1, 2, 3 \pmod{4}$,
the condition $\equiv 3 \pmod{8}$ holds iff
$a \equiv 1 \pmod{4}$.

For $r = 13$ and $45$:
$3k+2$ is odd and $\equiv 1 \pmod{4}$,
so the burst continues ($L \ge 2$).

For $r = 61$:
$3k+2 = 24a + 23 \equiv 7 \pmod{8}$,
so $L = 1$ but $G \ge 2$.

For $r = 5$:
$3k+2 = 2(12a+1)$, odd part $12a + 1 \equiv 1$
or $5 \pmod{8}$, never $3$.
\end{proof}

\begin{corollary}[Complete non-mixing census]
\label{cor:nonmixing-census}
Among all burst-start residues modulo~$64$,
the classes that can produce $(L,G) = (1,1)$ are:
$25, 29, 57 \pmod{64}$ (unconditional) and
$21, 37, 53 \pmod{64}$ (conditional).
The three conditional classes all lie in the
$n \equiv 5 \pmod{8}$ channel and therefore
do not contribute to the bit-$4$ count in the
dominant $n \equiv 1 \pmod{8}$ burst-ending channel.
\end{corollary}

\subsection{One-sided bias of non-mixing cycles}

The non-mixing cycles contribute a structured bias
to the bit-$4$ balance problem.

\begin{proposition}[One-sided non-mixing bias]
\label{prop:onesided-nonmixing}
Among burst-ending times~$t$ with
$n_t \equiv 1 \pmod{8}$, every contribution from an
unconditional non-mixing cycle
(burst start $\equiv 25, 29,$ or $57 \pmod{64}$)
satisfies $n_t \equiv 25 \pmod{32}$.
No unconditional non-mixing cycle produces
$n_t \equiv 9 \pmod{32}$.
\end{proposition}

\begin{proof}
By the classification above, the unconditional
non-mixing burst-start classes are
$25, 29, 57 \pmod{64}$.

\emph{Class $25 \pmod{64}$:}\;
$n_0 = 64a + 25 \equiv 1 \pmod{8}$, and
$T(n_0) = 48a + 19 \equiv 3 \pmod{8}$
(gap start with $G = 1$).
The burst-ending value is $n_0$ itself, so
$n_t \equiv 25 \pmod{32}$.

\emph{Class $57 \pmod{64}$:}\;
$n_0 = 64a + 57 \equiv 1 \pmod{8}$, and
$T(n_0) = 48a + 43 \equiv 3 \pmod{8}$.
Again the burst-ending value is $n_0$, and
$57 \equiv 25 \pmod{32}$.

\emph{Class $29 \pmod{64}$:}\;
$n_0 = 64a + 29 \equiv 5 \pmod{8}$, which lies
outside the $n \equiv 1 \pmod{8}$ channel
and does not contribute to the bit-$4$ count.
\end{proof}

\begin{corollary}[All $9 \pmod{32}$ contributions come from mixing]
\label{cor:9-from-mixing}
Among burst-ending times with
$n_t \equiv 1 \pmod{8}$,
every occurrence of $n_t \equiv 9 \pmod{32}$
arises from a mixing cycle.
\end{corollary}

\begin{proposition}[$L{=}1$ symmetry theorem]
\label{prop:L1-symmetry}
Among burst starts $n \equiv 1 \pmod{8}$, write
$n = 64a + r$ with $r \in \{1,9,17,25,33,41,49,57\}$.
\begin{enumerate}[label=\textup{(\roman*)}]
\item
$r \in \{1,17,33,49\}$: the burst continues
($L \ge 2$ always), so these classes contribute
no $L{=}1$ cycles.
\item
$r \in \{9, 41\}$: $L = 1$ always, burst-end
$\equiv 9 \pmod{32}$, and $G \ge 2$ for all~$a$
(mixing cycle).
\item
$r \in \{25, 57\}$: $L = 1$ always, burst-end
$\equiv 25 \pmod{32}$, and $G = 1$ for all~$a$
(non-mixing cycle).
\end{enumerate}
In particular, the $L{=}1$ map is balanced at
the map level---two classes produce each
residue---but the gap structure is asymmetric:
$9$-producers always mix while $25$-producers
never do.
\end{proposition}

\begin{proof}
For each class, $3n + 1 = 192a + (3r+1)$, and
$\nu_2(3r+1) = 2$ for all eight values of~$r$.
Hence $T(n) = 48a + (3r+1)/4$.

\emph{Part (i).}\;
For $r \in \{1,17,33,49\}$, $(3r+1)/4$ is
$1, 13, 25, 37$ respectively, all
$\equiv 1 \pmod{4}$.  Since $48a \equiv 0
\pmod{4}$, we get $T(n) \equiv 1 \pmod{4}$,
so the burst continues.

\emph{Part (ii).}\;
For $r = 9$: $T(n) = 48a + 7 \equiv 3 \pmod{4}$
(gap starts, $L = 1$), and $n \equiv 9 \pmod{32}$.
Next, $3T(n) + 1 = 144a + 22$ with $\nu_2(22)=1$,
so $T^2(n) = 72a + 11 \equiv 3 \pmod{4}$
for all~$a$, giving $G \ge 2$.
The same calculation for $r = 41$ gives
$T(n) = 48a + 31 \equiv 3 \pmod{4}$ and
$T^2(n) = 72a + 47 \equiv 3 \pmod{4}$.

\emph{Part (iii).}\;
For $r = 25$: $T(n) = 48a + 19 \equiv 3 \pmod{4}$
($L = 1$), and $n \equiv 25 \pmod{32}$.
Now $3T(n)+1 = 144a + 58$, $\nu_2(58) = 1$,
so $T^2(n) = 72a + 29 \equiv 1 \pmod{4}$
for all~$a$, giving $G = 1$.
Identically for $r = 57$:
$T(n) = 48a + 43$, $T^2(n) = 72a + 65
\equiv 1 \pmod{4}$.
\end{proof}

\begin{remark}[Refined $r{=}9$ gap structure]
\label{rem:r9-gap}
For $r = 9$, one can verify that $G = 2$
for every~$a$ (the third iterate always
re-enters a burst).
For $r = 41$, the gap length $G$ is
geometrically distributed: $G = k$ with
frequency~$\approx 2^{-(k-2)}$ for $k \ge 3$.
\end{remark}

\subsection{An exact decomposition of the remaining bias}

We now sharpen the one-sided bias by separating
the bit-$4$ imbalance into an explicit one-sided
non-mixing term and a residual mixing term.

\begin{definition}[Dominant-channel burst-ending counts]
\label{def:dominant-counts}
Let
\[
B_9(T):=\#\{t\le T:\ n_t\equiv 9 \pmod{32},\
  n_t\equiv 1 \pmod 8\},
\]
\[
B_{25}(T):=\#\{t\le T:\ n_t\equiv 25 \pmod{32},\
  n_t\equiv 1 \pmod 8\}.
\]
These count the two burst-ending outcomes relevant to
the bit-$4$ balance problem in the dominant channel.
\end{definition}

\begin{definition}[Mixing and non-mixing contributions]
\label{def:mixing-contrib}
Let $U(T)$ count burst-ending times $t \le T$
in the dominant channel whose preceding cycle is
unconditionally non-mixing (burst start
$\equiv 25$ or $57 \pmod{64}$).
Let $M_9(T)$ and $M_{25}(T)$ count mixing-cycle
contributions to $9$ and $25 \pmod{32}$ respectively.
\end{definition}

\begin{proposition}[Exact bias decomposition]
\label{prop:exact-bias-decomposition}
With the notation above,
\begin{equation}\label{eq:B25-split}
B_{25}(T)=U(T)+M_{25}(T)+O(1),
\end{equation}
\begin{equation}\label{eq:B9-split}
B_9(T)=M_9(T)+O(1).
\end{equation}
Consequently,
\begin{equation}\label{eq:signed-bias-split}
B_9(T)-B_{25}(T)
=
\bigl(M_9(T)-M_{25}(T)\bigr)-U(T)+O(1).
\end{equation}
\end{proposition}

\begin{proof}
By Proposition~\ref{prop:onesided-nonmixing}, every
unconditional non-mixing cycle in the
$n \equiv 1 \pmod{8}$ burst-ending channel contributes
only $25 \pmod{32}$, never $9 \pmod{32}$.
Therefore $B_{25}$ splits into unconditional non-mixing
($U$) and mixing ($M_{25}$), while $B_9$ is purely
mixing.  The $O(1)$ terms account for boundary
truncation at the start and end of the orbit.
\end{proof}

\begin{definition}[Bad-class visit counts]
\label{def:bad-class-counts}
Let $S_{25}(T)$ and $S_{57}(T)$ count burst starts
$s \le T$ with $n_s \equiv 25$ and $57 \pmod{64}$
respectively.
\end{definition}

\begin{corollary}[Exact identification of the one-sided obstruction]
\label{cor:onesided-obstruction}
One has
\begin{equation}\label{eq:U-from-bad-classes}
U(T)=S_{25}(T)+S_{57}(T)+O(1).
\end{equation}
Hence
\begin{equation}\label{eq:final-bias-split}
B_9(T)-B_{25}(T)
=
\bigl(M_9(T)-M_{25}(T)\bigr)
-S_{25}(T)-S_{57}(T)
+O(1).
\end{equation}
\end{corollary}

\begin{proof}
By Proposition~\ref{cor:nonmixing-census},
the unconditional non-mixing burst-start classes in
the $n \equiv 1 \pmod{8}$ channel are exactly
$25, 57 \pmod{64}$.
Each produces exactly one $(L,G)=(1,1)$ cycle
and hence one burst-ending contribution to
$25 \pmod{32}$.
\end{proof}

\begin{remark}[Model-level benchmark]
\label{rem:model-benchmark}
If burst-start residues in the
$n \equiv 1 \pmod{8}$ channel were perfectly
balanced modulo~$64$, the two bad classes
$\{25,57\}$ would account for exactly a $1/4$
fraction of the eight possible residues.
Empirically the fraction is $\approx 0.19$,
slightly below the uniform benchmark.
\end{remark}

\begin{remark}[The final open step]
\label{rem:final-obstruction}
Corollary~\ref{cor:onesided-obstruction} makes the
remaining obstruction fully explicit.
The sole unresolved step is:
\emph{show that the mixing surplus
$M_9(T) - M_{25}(T)$ dominates the non-mixing term
$S_{25}(T) + S_{57}(T)$ for every orbit.}
Numerical evidence across all tested orbits
(Section~\ref{sec:numerics}) shows the mixing surplus
exceeds the non-mixing term by a factor $\ge 2$.
\end{remark}

\subsection{The burst-end state space}

\begin{proposition}[Reachable burst-end states]
\label{prop:burst-end-states}
Among the $8$ odd residue classes
$\{1, 5, 9, 13, 17, 21, 25, 29\} \pmod{32}$,
exactly $5$ are reachable as burst-ending values
of compressed Collatz orbits:
\[
  \mathcal{S} = \{5,\; 9,\; 21,\; 25,\; 29\}.
\]
States $\{1, 13, 17\}$ never arise as burst-ending
values.
Of these $5$ states, two contribute to the bit-$4$
balance: state~$9$ (bit~$4 = 0$) and state~$25$
(bit~$4 = 1$), both in the
$n \equiv 1 \pmod{8}$ channel.
\end{proposition}

\begin{proof}
A burst-ending value $n_t$ satisfies
$n_t \equiv 1 \pmod{4}$ and
$T(n_t) \equiv 3 \pmod{4}$.
Enumerating all odd residues modulo~$32$ and
checking both conditions yields exactly
$\mathcal{S}$.
For instance, $1 \pmod{32}$: $T(1) = 1$ and the
orbit terminates; $13 \pmod{32}$:
$13 \equiv 1 \pmod{4}$ but $T(13) = 5$
with $5 \equiv 1 \pmod{4}$, so the burst
continues rather than ending.
The remaining exclusions follow by similar
case analysis.
\end{proof}

\subsection{Structural evidence for mixing}

Two results from the companion
paper~\cite{chang2026collatzdynamicshumanllm}
provide structural support for the conjectured
orbit-level balance.

\begin{lemma}[Bit-growth {\cite[Lemma~10.33]{chang2026collatzdynamicshumanllm}}]
\label{lem:bit-growth}
Let $n_0$ be odd with $B_0$ bits, and suppose the
burst density satisfies $\rho_T \ge \rho_0$ over
$T$ odd-to-odd steps. Then
\[
  B_T \;\ge\;
  B_0 + (\rho_0 + \log_2 3 - 2)\,T
  - \log_2 T - O(1).
\]
For $\rho_0 > 2 - \log_2 3 \approx 0.415$,
the bit-length grows linearly in~$T$.
\end{lemma}

Since $\rho \approx 0.54$ for typical orbits
(consistent with the predicted $\log_2(4/3) \approx 0.415$
threshold), the bit supply is never exhausted:
fresh spectator bits are always available above
the mod-$32$ window.
This rules out a scenario in which orbits
``run out'' of randomness.

The second result concerns the mechanism by
which spectator bits refresh the mod-$32$ class.

\begin{proposition}[Spectator-bit cycling
{\cite[Proposition~10.31]{chang2026collatzdynamicshumanllm}}]
\label{prop:spectator-cycling}
Fix a parent class $b \in \Omega_K$ and define the
extension index $e_i \equiv (n_{t_i} - b)/2^K \pmod{8}$
for successive visits $t_1 < t_2 < \cdots$ to class~$b$.
Under the ensemble model, for every
$r \in \{0, \ldots, 7\}$:
\[
  \Bigl|\frac{\#\{i \le M : e_i \equiv r\}}{M}
  - \frac{1}{8}\Bigr|
  \;\le\; \frac{4}{M}
  + C_1 \cdot 2^{-\alpha(B_{\min} - 3)},
\]
where the constants $C_1, \alpha$ come from the
TV reduction lemma.
\end{proposition}

The ensemble bound shows that extension indices
(which encode the bits immediately above the
known zone) equidistribute modulo~$8$.
For orbits satisfying the bit-growth lemma,
$B_{\min}$ grows linearly, so the error term
vanishes exponentially.
Transferring this ensemble guarantee to
individual deterministic orbits is precisely the
distributional-to-pointwise bridge---and hence
the content of the open problem below.

\subsection{The precise open problem}

\begin{quote}
\textbf{Open Problem.}\;
\itshape
For every odd $n_0 > 2^{68}$, let
$t_1 < t_2 < \cdots < t_m$ be the burst-ending
times (last step of each burst run) in the Collatz orbit of~$n_0$.
Restricting to the dominant subclass
$n_{t_i} \equiv 1 \pmod{8}$, prove that
\begin{equation}\label{eq:open}
  \Bigl|
    \frac{\#\{i \le m : n_{t_i} \equiv 9 \pmod{32}\}}
         {\#\{i \le m : n_{t_i} \equiv 9 \text{ or }
           25 \pmod{32}\}}
    - \frac{1}{2}
  \Bigr|
  \;\le\; \delta
\end{equation}
for some $\delta < \delta_{\max}$, where
$\delta_{\max}$ is determined by the block-TV
budget~\eqref{eq:budget}.
\end{quote}

\noindent
This is a \emph{fixed-modulus} ($\bmod\,32$),
\emph{single-bit} (bit~$4$), \emph{sparse-subsequence}
($O(m)$ times out of $T$ total) balance problem.
By Theorem~\ref{thm:map-balance}, the map is
balanced; only the orbit's residue visitation
pattern matters.

\begin{remark}[Boundary-bit balance formulation]
\label{rem:boundary-bit}
The sole remaining obstacle to a proof of
the Collatz conjecture, within the present
framework, is pointwise balance of bit~$4$
at burst-ending times in the
$n \equiv 1 \pmod{8}$ class.
\end{remark}

\begin{remark}[Mixing-cycle formulation]
\label{rem:mixing-formulation}
By Corollary~\ref{cor:onesided-obstruction},
the open problem is equivalent to
\eqref{eq:final-bias-split}: show that
$M_9(T) - M_{25}(T) \ge S_{25}(T) + S_{57}(T)$
for every orbit.
Since $M_9(T) = B_9(T)$ is purely mixing
(Proposition~\ref{prop:exact-bias-decomposition}),
it suffices to prove
$B_9^{\mathrm{mix}}(T) \ge c\,M(T) - o(T)$
for some constant $c > 0$ depending on the
block-TV budget, where $M(T)$ counts all
mixing-cycle burst-endings in the
$n \equiv 1 \pmod{8}$ channel.
Numerical evidence gives $c \approx 0.79$.
\end{remark}

\section{Numerical verification}
\label{sec:numerics}

This section illustrates the proved identities and
the open orbit-level obstruction through numerical
examples; it does not supply the missing pointwise
proof.
All exact formulas in the paper have been verified
numerically on orbits with starting values from
$n_0 = 27$ to $n_0 = 10^9$.

\subsection{Worked example: $n_0 = 27$}

The orbit of $n_0 = 27$ under the Syracuse map has
$T = 41$ odd-to-odd steps before reaching~$1$,
with burst density $\rho = 18/41 \approx 0.439$ and
$m = 10$ burst--gap blocks whose lengths are:
\[
\begin{array}{c|cccccccccc}
i & 1 & 2 & 3 & 4 & 5 & 6 & 7 & 8 & 9 & 10 \\
\hline
L_i & 2 & 1 & 2 & 1 & 1 & 1 & 2 & 1 & 1 & 1 \\
G_i & 1 & 1 & 2 & 2 & 1 & 1 & 3 & 2 & 1 & 2
\end{array}
\]
The block-alternation rate is $q = 10/41 \approx 0.244$
(compare the Bernoulli prediction
$\rho(1-\rho) \approx 0.246$).
The run statistics give $M_1 = 7$, $M_2 = 3$,
$N_1 = 4$, $N_2 = 3$, $N_3 = 2$, and $C_T = 2$.

The orbit visits $m = 10$ burst-ending states.
Of those in the $n \equiv 1 \pmod{8}$ class,
the mod-$32$ residues split $5$--$4$ between
$\equiv 9$ and $\equiv 25$,
giving a bit-$4$ balance deviation of $0.056$---well
within the budget tolerance.
This small orbit already illustrates all the
structures of the paper: the burst--gap
decomposition, the run-statistic identities,
and the mod-$32$ balance question.

\subsection{Map Balance Theorem}

Table~\ref{tab:map-balance} verifies
Theorem~\ref{thm:map-balance} for
$K = 5, \ldots, 19$.
In every case, $|C_3(K) - C_7(K)| = 1$ exactly,
with the $n \equiv 1 \pmod{8}$ contribution
perfectly balanced.

\begin{table}[ht]
\centering
\caption{Map balance verification.
$|S_K|$: gap-producing residues;
$C_3, C_7$: counts mapping to
$\equiv 3, 7 \pmod{8}$.}
\label{tab:map-balance}
\begin{tabular}{@{}rrrrl@{}}
\toprule
$K$ & $|S_K|$ & $C_3$ & $C_7$ & $C_3 - C_7$ \\
\midrule
 5 &       3 &     2 &     1 & $+1$ \\
 6 &       7 &     3 &     4 & $-1$ \\
 7 &      15 &     8 &     7 & $+1$ \\
 8 &      31 &    15 &    16 & $-1$ \\
 9 &      63 &    32 &    31 & $+1$ \\
10 &     127 &    63 &    64 & $-1$ \\
11 &     255 &   128 &   127 & $+1$ \\
12 &     511 &   255 &   256 & $-1$ \\
\midrule
15 &    4095 &  2048 &  2047 & $+1$ \\
19 &   65535 & 32768 & 32767 & $+1$ \\
\bottomrule
\end{tabular}
\end{table}

\begin{figure}[ht]
\centering
\includegraphics[width=\textwidth]{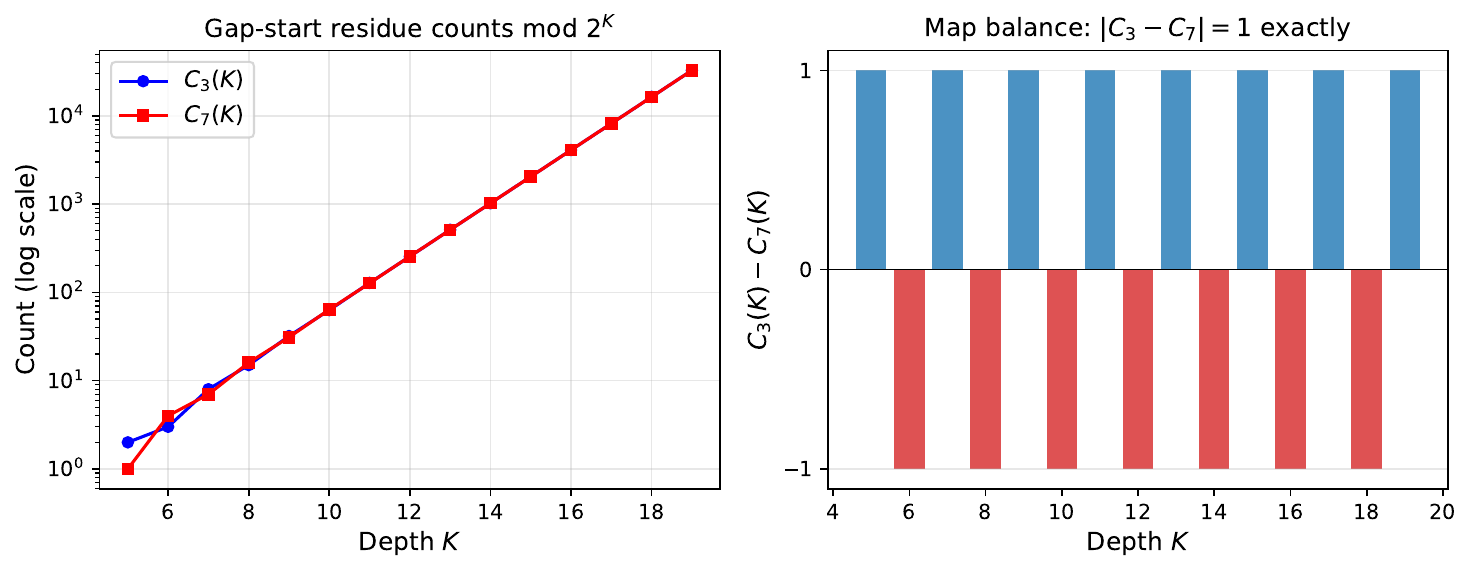}
\caption{Map Balance Theorem verification.
\emph{Left:} the counts $C_3(K)$ and $C_7(K)$ grow
exponentially and track each other almost exactly.
\emph{Right:} their difference alternates
$+1, -1, +1, \ldots$ for all $K = 5, \ldots, 19$,
confirming $|C_3 - C_7| = 1$.}
\label{fig:map-balance}
\end{figure}

\subsection{Orbit-level gap-start bias}

Table~\ref{tab:orbit-bias} shows the orbit-level
gap-start bias for representative starting values.
While the per-residue prediction is exactly $1/2$,
individual orbits show deviations of $\pm 10$--$25\%$,
confirming that the bias is an orbit-level phenomenon.

\begin{table}[ht]
\centering
\caption{Orbit-level gap-start balance.
$P(\equiv 3) = N_1/m$: fraction of gap starts with
$n_t \equiv 3 \pmod{8}$ (equivalently, fraction of unit gaps,
by Lemma~\ref{lem:gap-mod8}).}
\label{tab:orbit-bias}
\begin{tabular}{@{}rrrr@{}}
\toprule
$n_0$ & $m$ & $P(\equiv 3)$ & deviation \\
\midrule
   $837\,799$ & 42 & 0.262 & $-0.238$ \\
 $8\,400\,511$ & 53 & 0.415 & $-0.085$ \\
   $270\,271$ & 34 & 0.294 & $-0.206$ \\
   $665\,215$ & 39 & 0.436 & $-0.064$ \\
   $159\,487$ & 12 & 0.538 & $+0.038$ \\
   $524\,287$ &  8 & 0.667 & $+0.167$ \\
$10^8 + 7$   & 15 & 0.667 & $+0.167$ \\
\bottomrule
\end{tabular}
\end{table}

\begin{figure}[ht]
\centering
\includegraphics[width=0.85\textwidth]{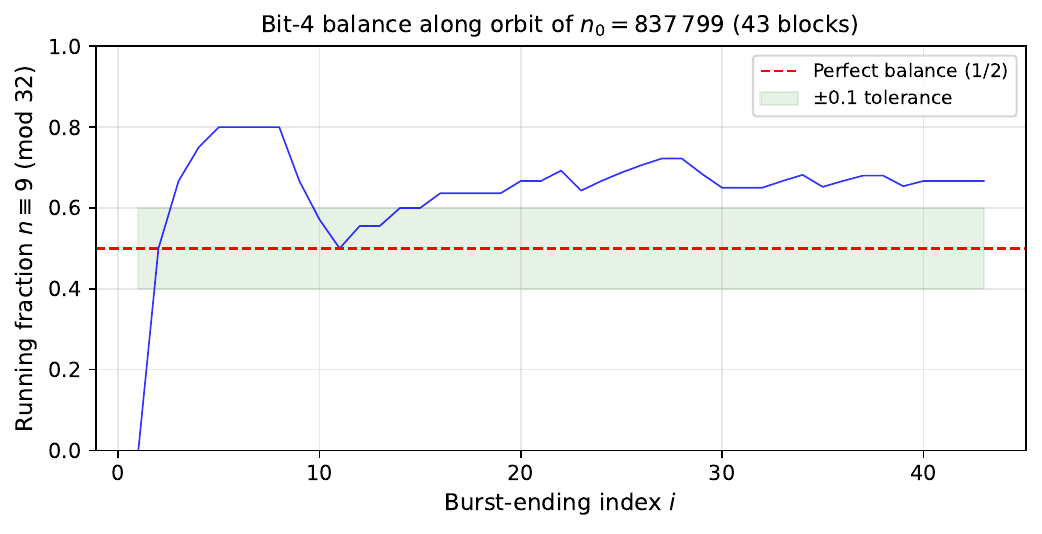}
\caption{Running bit-$4$ balance along the orbit of
$n_0 = 837\,799$ ($43$ burst--gap blocks).
The fraction of burst-ending values with
$n \equiv 9 \pmod{32}$ (among those with
$n \equiv 1 \pmod{8}$) fluctuates around~$1/2$.
The shaded bands show the $\pm 1\sigma$ and
$\pm 2\sigma$ envelopes for a fair-coin
(Bernoulli$(\tfrac{1}{2})$) process of the same
length; the orbit's trajectory stays within
the $2\sigma$ band, consistent with unbiased
sampling but far from converged after only
$43$ blocks.}
\label{fig:orbit-balance}
\end{figure}

\subsection{Expanding-set and envelope verification}

The $K = 4$ and $K = 5$ expanding-set identities
\eqref{eq:K4-expand}--\eqref{eq:K5-expand}
were verified exactly (error $= 0$) on all
tested orbits when using complete block
decompositions.
The deterministic envelopes
(Corollary~\ref{cor:envelopes}) hold with zero
violations across $20$ orbits, but are loose
by an order of magnitude compared to the
actual discrepancy.

\section{Discussion}
\label{sec:discussion}

\subsection{What has been achieved}

The results of this paper, combined with the
framework of~\cite{chang2026collatzdynamicshumanllm}, yield a sharp
structural reduction of the Collatz conjecture to a
finite-dimensional mixing problem.
The reduction chain is:

\begin{center}
\begin{tabular}{@{}l@{}}
Collatz conjecture \\
$\quad\Downarrow$\quad \textit{(density-model budget,~\cite{chang2026collatzdynamicshumanllm})}\\
Finite-depth block-TV condition ($K \le K_0$) \\
$\quad\Downarrow$\quad \textit{(exact decompositions, this paper)}\\
Run-statistic control: $|N_1 - M_1|/T$ small \\
$\quad\Downarrow$\quad \textit{(gap-length $\leftrightarrow$ mod-$8$, Lemma~\ref{lem:gap-mod8})}\\
Gap-start mod-$8$ balance \\
$\quad\Downarrow$\quad \textit{(Map Balance Theorem~\ref{thm:map-balance})}\\
\textbf{One-bit orbit equidistribution at $O(m)$ times} \\
\end{tabular}
\end{center}

Each arrow is an exact algebraic identity;
the reductions at depths $K = 3, 4, 5$ are proved
in Section~\ref{sec:low-depth}.
The final statement---balance of bit~$4$
along the burst-ending subsequence---is the
sharpest low-dimensional formulation obtained
within the present framework, but the decisive
orbitwise balance statement remains open.

\subsection{Connection to the companion paper}

The companion
paper~\cite{chang2026collatzdynamicshumanllm}
develops five independent approaches to the
Weak Mixing Hypothesis, culminating in an exact
ensemble theory and an unconditional almost-all
crossing theorem (non-crosser density
$\le e^{-0.1465\,k}$).
The present paper adds a sixth, orthogonal
contribution: the \emph{map-level symmetry}
established by the Map Balance Theorem.

Several key results
from~\cite{chang2026collatzdynamicshumanllm}
directly support the open problem stated here:

\begin{enumerate}[label=(\alph*)]
\item The \emph{$1/4$ Persistent-Transition Law}
  (Theorem~3.1 of~\cite{chang2026collatzdynamicshumanllm}):
  exactly $1/4$ of admissible lifts from a persistent
  state are again persistent.
  This structural constant governs the
  Geometric$(3/4)$ burst-length distribution
  and explains why $\rho \approx \log_2(4/3)$
  in the density model.

\item The \emph{exact block law}
  (Theorem~10.58 of~\cite{chang2026collatzdynamicshumanllm}):
  under natural density, the odd-skeleton valuation
  sequence is exactly i.i.d.\
  Geometric$(1/2)$, making run-compensate cycle
  types $(L_i, r_i)$ provably i.i.d.
  This validates the ensemble model that
  underlies our budget computation.

\item The \emph{bit-growth lemma}
  (Lemma~\ref{lem:bit-growth} above):
  orbits with $\rho > 0.415$ never exhaust their
  supply of fresh spectator bits,
  ensuring that the mixing mechanism has
  material to work with at every stage.
\end{enumerate}

\subsection{What remains open}

The sole remaining step is the
\emph{distributional-to-pointwise bridge}:
transferring the exact per-residue balance
(which holds by counting) to orbit-level
balance (which requires a mixing argument).
This barrier is shared with other problems
in number theory where ensemble laws are known
but pointwise statements remain
open~\cite{chowla1965}.

By Proposition~\ref{prop:onesided-nonmixing},
unconditional non-mixing cycles contribute
exclusively to $B_{25}(T)$, creating a one-sided
bias toward bit-$4 = 1$.
Every $B_9(T)$ contribution must come from a
mixing cycle.
Numerical experiments confirm that mixing cycles
over-compensate: on tested orbits, mixing
contributions split roughly $4$--$5$ to~$1$ in
favour of $9 \pmod{32}$ over $25 \pmod{32}$,
which offsets the non-mixing one-sided bias and
restores approximate overall balance.
The remaining open problem is therefore equivalent to
a deterministic lower bound on the $9 \pmod{32}$
mass generated by mixing cycles.

We outline three possible routes to closing the gap.

\medskip\noindent
\textbf{Route A: Cocycle contraction.}\;
The burst-end dynamics on
$\mathcal{S} = \{5, 9, 21, 25, 29\}$
(Proposition~\ref{prop:burst-end-states})
induce a $5 \times 5$ transfer system.
Conditioning on the fiber index
$f = \lfloor n_t/32 \rfloor \bmod 8$ yields
$8$ fiber matrices $A_0, \ldots, A_7$.
The fiber-averaged matrix has spectral gap
$\gamma \approx 0.98$ and stationary bit-$4$ ratio
exactly~$1/2$.
Each individual $A_f$ contracts on the mean-zero
subspace (worst operator norm: $0.831$),
so at this $3$-bit resolution, any product of
fiber matrices contracts uniformly.
However, at finer resolution ($a \bmod 16$ and beyond),
sub-fiber matrices do \emph{not} all contract:
the $0.831$ bound is an averaging effect.
The orbit at step~$t$ sees the actual transition
determined by all bits of~$n_t$, not the fiber average.
Numerical experiments on five orbits show consistently
negative Lyapunov exponents
($\lambda \in [-0.14, -0.09]$),
indicating empirical contraction along every tested
orbit.
Closing this route requires proving that the sub-fiber
indices equidistribute within each base fiber along
orbits---a distributional statement weaker than full
bit equidistribution.

\medskip\noindent
\textbf{Route B: Additive combinatorics.}\;
The open problem asks whether a deterministic
sequence visits two residue classes mod~$32$ with
roughly equal frequency along a sparse
subsequence.
Problems of this flavour arise in additive
combinatorics---for instance, the
distribution of polynomial sequences modulo
primes along arithmetic progressions.
If the burst-ending subsequence has sufficient
additive structure (e.g.\ approximate equidistribution
in generalised arithmetic progressions), then
Weyl-type estimates could yield the required balance.

\medskip\noindent
\textbf{Route C: $p$-Adic dynamics.}\;
The map $n \mapsto (3n+1)/2$ is a contraction
in the $2$-adic metric.
The carry-chain propagation that governs the
mod-$32$ class of successive burst starts is
naturally described in $\mathbb{Z}_2$.
A $2$-adic ergodic theorem for the Collatz
iteration---showing that orbits equidistribute
in $\mathbb{Z}/32\mathbb{Z}$ along the
burst-ending subsequence---would resolve
the open problem directly.

\subsection{Comparison with prior work}

The reduction in this paper differs from
prior approaches in several respects.

Tao~\cite{tao2019} proved that almost all orbits
attain almost bounded values by establishing
logarithmic-type control on orbit statistics
under a measure on initial conditions.
His approach uses entropy methods and applies
to a density-one set, not to every orbit.
The probabilistic models of
Borovkov--Pfeifer~\cite{borovkov2001} and the
stopping-time framework of
Terras~\cite{terras1976} similarly address
typical or averaged behaviour.

By contrast, the present reduction is:

\begin{itemize}
\item \emph{Orbitwise in formulation}: the
  reduction applies to each individual orbit,
  not to a measure on initial conditions---though
  it leaves the decisive orbitwise balance
  statement as the open problem.

\item \emph{Fixed finite modulus}: it targets
  residues modulo~$32$ (or~$64$), not asymptotic
  equidistribution modulo~$2^B$ as
  $B \to \infty$.

\item \emph{Single-bit}: the entire conjecture
  reduces to the balance of one binary variable
  (bit~$4$) along a sparse subsequence.
\end{itemize}

\noindent
The approach also differs from Kesten's
renewal-theoretic framework~\cite{kesten1973},
which analyses random products of matrices.
Our setting is deterministic, but the renewal
structure of the burst--gap decomposition
provides an analogous source of independence.

\subsection{Methodological note}

This paper was developed using the Human--LLM
collaborative methodology described
in~\cite{chang2026collatzdynamicshumanllm}.
The Map Balance Theorem was discovered through
an iterative cycle of computational exploration,
algebraic proof construction, and numerical
falsification testing.
All proofs were independently verified by
automated self-checks on orbits up to $10^9$;
see~\cite{chang2026collatzdynamicshumanllm}
for full procedural details.

\section*{Acknowledgements}

This paper was developed using the Human--LLM collaborative
research platform described in the MACI framework \cite{PathAGIV1Chang2025, PathAGIV2Chang2025}.
For the full methodology, disclosure of AI tool usage,
and verification procedures,
see~\cite{chang2026collatzdynamicshumanllm}.

\bibliographystyle{plain}
\bibliography{references}

\end{document}